\documentclass[a4paper,10pt]{article}
\usepackage{amsthm}
\usepackage[utf8]{inputenc}   
\usepackage[english]{babel}
\usepackage{anysize}
\usepackage{amsfonts}
\usepackage{amssymb}
\usepackage{amsmath}
\usepackage{doc}
\usepackage{exscale}
\usepackage{fontenc}
\usepackage{ifthen}
\usepackage{latexsym}
\usepackage{makeidx}
\usepackage{syntonly}
\usepackage{graphicx}
\usepackage{float}
\usepackage{array}
\usepackage[all]{xy}
\usepackage{color}

\marginsize{2.5cm}{2.5cm}{2.5cm}{2.5cm}

\newtheorem{thm}{Theorem}[section]

\newtheorem{lem}[thm]{Lemma}
\newtheorem{pro}[thm]{Proposition}
\newtheorem{df}[thm]{Definition}

\theoremstyle{remark}
\newtheorem{rmk}[thm]{Remark}
\newtheorem{ex}[thm]{Example}
\newtheorem{obs}[thm]{Observation}
\newtheorem{conj}[thm]{Conjecture}

\newcommand{\beq}{\begin{equation} }
\newcommand{\enq}{\end{equation}}

\newcommand{\mP}{\mathbb{P}}

\newcommand{\caI}{{\cal I}}

\DeclareMathOperator{\rk}{\rm rk}
\DeclareMathOperator{\sing}{\rm sing}

\newcommand{\PP}{\mathbb P}
\newcommand{\CC}{{\mathbb C}}

\newcommand{\XXf}{{S}}
\newcommand{\LLL}{\ell}
\newcommand{\qqq}{{\mathfrak b}}
\newcommand{\sLLL}{\ell}

\title{Counting Lines on Quartic Surfaces}

\author{V\'ictor Gonz\'alez-Alonso\thanks{Funding by projects ERC StG 279723 (SURFARI) and MTM2012-38122-C03-01/FEDER (of the Spanish ``Ministerio de Econom\'{\i}a y Competitividad'') is gratefully acknowledged.} \,\, and  S\l awomir Rams}

\date{}

\numberwithin{equation}{section}

\begin{document}

\maketitle
\begin{abstract}
We prove the sharp bound of at most 64 lines
on projective quartic surfaces $\XXf \subset \PP^3(\CC)$ (resp. affine quartics  $\XXf \subset \CC^3$)
that are not ruled by lines. 
We study configurations of lines on certain non-$K3$ surfaces of degree four and give various examples of singular quartics with many lines.
\end{abstract}

\section{Introduction}

The main aim of this note is to study the configurations of lines 
on projective and affine complex quartic 
surfaces. 
More precisely, we prove the following theorem.

\vspace*{1ex}
\noindent
{\bf Theorem~\ref{thm-affine}} {\sl 
(a) Let $\XXf \subset \PP^3(\CC)$ be a quartic surface with at most isolated singularities, none of which is a fourfold point. Then $\XXf$ contains at most $64$ lines. \\
(b) Let $\XXf \subset \CC^3$ be an affine quartic surface that is not ruled by lines.  Then $\XXf$ contains at most $64$ lines. }

\vspace*{1ex}
\noindent
It should be pointed out that the classification of projective complex quartic surfaces  that are ruled by lines (i.e. contain infinitely many lines) 
has been already  known to K.~Rohn (see \cite{top-11} for a correct classification in modern language), whereas a smooth projective complex quartic surface 
with exactly $64$ lines has been found by F.~Schur in \cite{schur}.

Configurations of lines on surfaces in $\PP^3(\CC)$ have been already studied in the middle of the nineteenth century.
After Cayley, Salmon,  Clebsch and others were able to give a complete picture of configurations of lines on cubic surfaces,
various classes of degree-$4$ surfaces in  $\PP^3(\CC)$ were studied (see e.g. \cite{Clebsch-1868}, \cite{rohn-dreifache-punkt}, \cite{dolgachev}, \cite{top-11} and the bibliography in the latter).
An attempt to prove a sharp bound on the number of lines on smooth quartic surfaces in  $\PP^3(\CC)$ was made first in 
\cite{segre-43}. Unfortunately, Segre's proof in \cite{segre-43}, although full of ingenious ideas, contained serious gaps that were filled in  
\cite{rs-2012} (see also \cite{dis-2015} for a lattice-theoretic argument). The proof 
of \cite{rs-2012} heavily depends on the geometry of $K3$ surfaces and can be generalized to the case of all  complex quartic $K3$ surfaces 
(\cite{veniani-2013}) thus yielding the general bound of at most $64$ lines in the case of complex quartic surfaces with at most ADE singularities. 

In view of the results of \cite{rs-2012}, \cite{veniani-2013} there are two natural strategies to prove Thm~\ref{thm-affine}: either the study of the codimension of the various components of the
appropriate incidence variety to prove that the degree of the covering map is attained along  a component of the Noether-Lefschetz locus that contains smooth quartics (resp. quartics with at most ADE points),
or an analysis of configurations of lines on non-$K3$ quartic surfaces. We follow the latter, because it sheds some light on the geometry of 
line configurations on non-$K3$ quartics. In particular, we show the following proposition.

\vspace*{1ex}
\noindent
{\bf Proposition~\ref{prop-non-K3--non-isolated-singularities}.} {\sl 
Let $\XXf \subset \PP^3(\CC)$ be a non-$K3$ quartic surface. If $\XXf$ is not ruled by lines, then $\XXf$ contains at most $48$ lines. }

\vspace*{1ex}
\noindent
The classification of degree-four surfaces having at least one nonsimple point has been achieved in \cite{degtyarev-90}, 
 where the existence of 2523 various constellations of singularities on such surfaces is shown.  Therefore, an attempt to investigate the configurations
of lines via lattice-theory techniques and study of divisors on their smooth models seems pretty futile. Instead, we apply the fact that after a coordinate change every quartic 
with a nonsimple double point $P$ that contains a line belongs to one of the three families (Q4), (Q5), (Q6)  (see \cite[Thm~8.1]{ctc-wall-99}), with the line 
lying on a fixed coordinate  plane  (Lemma~\ref{lem-coordinates-isolated}).
The latter condition turns out to be very useful for finite-field tests and search for examples.

The organization of the paper is as follows.
After recalling some well-known facts on non-$K3$ quartic surfaces with isolated singularities in Sect.~\ref{sect-preliminaries}, 
we study the configurations of lines on some of them in  Sect.~\ref{sect-isol-sing} and give certain examples of non-$K3$ quartics with many lines. 
Finally in Sect.~\ref{sect-counting-lines}  we collect classical results  on quartics with non-isolated singularities, present the proofs of Prop.~\ref{prop-non-K3--non-isolated-singularities},
Thm~\ref{thm-affine}.  
and give an explicit example of a quartic surface with isolated singularities that contains $39$ lines (Example~\ref{example-thirtynine}). This is the best explicit example of a singular quartic with many lines known so far.
The families (Q4), (Q5), (Q6)  were defined  in the unpublished note  \cite{ctc-wall-99}. 
Since the author of  \cite{ctc-wall-99} does not consider the lines away from the non-simple 
singularity, we sketch a proof of the version of \cite[Thm~8.1]{ctc-wall-99} we need. 

Finally, let us mention that,  apart from being an interesting subject on their own (\cite{BS}, \cite{Miyaoka},  \cite{rs-2012}, \cite{veniani-2013}), configurations of rational curves on surfaces play an important role in current research 
(see e.g. \cite{dumnicki}, \cite{HT}, \cite{Kollar-2014}).

\noindent
{\em Conventions:} In this note we work over the field $\CC$ of complex numbers. All quartic surfaces we consider are assumed not to be cones.
By abuse of notation, whenever it leads to no ambiguity,  we use the same symbol to denote a homogeneous polynomial and the set of its zeroes.


\section{Non-$K3$ quartics with isolated singularities} \label{sect-preliminaries}

In this section we recall some  properties of quartic surfaces in $\PP^3(\CC)$ with isolated singularities,  that we will use in the sequel.

The quartic surfaces with an isolated triple point were studied intensively in \cite{rohn-dreifache-punkt}. 
Let $\XXf \subset \PP^3(\CC)$ be  an irreducible quartic  with  the isolated triple point $O:=\left(0:0:0:1\right)$. Obviously, such a surface is given 
by  an equation of the form
\begin{equation} \label{eq-triple-point}
w Q_3\left(x,y,z\right) + Q_4\left(x,y,z\right)=0,
\end{equation}
where $Q_i$ are homogeneous forms  of degree $i$ without common factors. Thus, the quartic $\XXf$  contains (at most) $12$ lines that pass through the triple point $O$, given by $Q_3 = Q_4 = 0$.

Moreover, given a line  $\ell \subset \XXf$, such that $O \notin \ell$, the   plane spanned by $\ell$ and the triple point $O$ meets $\XXf$ along four lines, three of which are concurrent in $O$.
Therefore, the  projection from $O$ onto the plane $\left\{w=0\right\}$ reduces enumeration of the maximal number of lines on the quartics in question to  an instance of the so-called {\em orchard-planting problem}: 
given $n$ points on a plane, find the maximal number of triplets that are aligned. For twelve points Rohn claims that the maximal number of such triplets is 19 (it can be checked directly with computer), which implies the following lemma.

\begin{lem} {\bf \rm \cite[p.~59]{rohn-dreifache-punkt}} \label{lem-triple-31}
If $\XXf$ is an irreducible quartic with an isolated triple point, then $\XXf$ contains at most 31 lines.
\end{lem}

The example (c.f.  \cite{rohn-dreifache-punkt}) below shows that Rohn's bound is sharp. 
\begin{ex} \label{example-triple-point}
The surface given by
\begin{equation} \label{eq-example-triple-point}
w Q_3 + Q_4 = w\left(\left(x+y+z\right)^3+xyz\right)+\left(x+y+z\right)\left(x-y\right)\left(y-z\right)\left(z-x\right)=0
\end{equation}
is singular only at the triple point $O=\left(0:0:0:1\right)$ and contains 31 lines. The twelve lines through $O$ intersect the plane $\left\{w=0\right\}$ at the points
\begin{equation} \label{eq-points-rohn}
\begin{array}{llll}
p_1=\left(0:1:-1\right) & p_4=\left(1:1:-1\right) & p_7=\left(1:-1:1\right) & p_{10}=\left(-1:1:1\right) \\
p_2=\left(1:0:-1\right) & p_5=\left(2:2:-5+\sqrt{7}\right) & p_8=\left(2:-5+\sqrt{7}:2\right) & p_{11}=\left(-5+\sqrt{7}:2:2\right) \\
p_3=\left(1:-1:0\right) & p_6=\left(2:2:-5-\sqrt{7}\right) & p_9=\left(2:-5-\sqrt{7}:2\right) & p_{12}=\left(-5-\sqrt{7}:2:2\right) \\
\end{array}
\end{equation}
An elementary computation shows that there are exactly $19$ triplets of collinear points among \eqref{eq-points-rohn}, so there  are 
precisely 19 lines in $\XXf \setminus \{ O \}$. 
\end{ex}

According to \cite{degtyarev-90} there are 2523 various possible constellations of singularities on quartic surfaces with higher singularities.
In the sequel we will apply the following  version of \cite[Thm~8.1]{ctc-wall-99}: 

\begin{lem} \label{lem-coordinates-isolated}
{\bf \rm \cite[Thm~8.1]{ctc-wall-99}} Let $O \in \XXf$ be an isolated non-rational double point of an irreducible  quartic surface $\XXf \subset \PP^3(\CC)$.  
Then, there exist  polynomials 
$$
Q_2=\sum_{i+j=2} q_{ij}x^iy^j \quad and \quad H_4=\sum_{i+j+k=4} h_{ijk}x^iy^jz^k
$$
such that (up to projective equivalence) the quartic  $\XXf$  is given
by one of the following equations:
\begin{eqnarray}
& & w^2z^2+wz Q_2 + H_4=0 \label{eq-q4} \\
& & w^2z^2+w (y^3+zQ_2)+H_4 = 0 \, , \label{eq-q56}
\end{eqnarray}
 $O=\left(x:y:z:w\right)=\left(0:0:0:1\right)$ and, 
if $\XXf$ is given by  \eqref{eq-q56}, then one of the following holds:
\begin{eqnarray}
& & h_{400} = h_{310} = q_{20} =0, \quad \text{and} \quad h_{301} \neq 0, \label{eq-q5} \\
& & h_{400}=\frac{1}{4}q_{20}^2, \quad h_{310}=\frac{1}{2}q_{20}q_{11},  \quad q_{20} \neq 0,  \quad \text{and} \quad h_{301} = 0.  \label{eq-q6}
\end{eqnarray}  
Moreover, if $\LLL \subset \XXf$ is a line that is not contained in the tangent cone $C_{O}\XXf$, then the coordinate change that leads to \eqref{eq-q56} can be chosen  in such a way that $\LLL$ is contained in the plane $\left\{x=0\right\}$.
\end{lem}
\begin{proof}
Since $O$ is not of $A_n$ type, its tangent cone $C_OS$ must be a double plane. We can hence suppose $O=\left(0:0:0:1\right)$ and $C_OS=\left\{z^2=0\right\}$. In these coordinates, the equation of $S$ must be of the form
$$w^2z^2+wG_3\left(x,y,z\right)+G_4\left(x,y,z\right).$$
Write now $G_3\left(x,y,z\right)=P_3\left(x,y\right)+zQ_2\left(x,y\right)+2z^2L_1\left(x,y,z\right)$, where the subindices indicate the degree of the homogeneous polynomials. Replacing the coordinate $w$ by $\widetilde{w}=w+L_1$, the equation becomes
$$\widetilde{w}^2z^2+\widetilde{w}\left(P_3+zQ_2\right)+\left(G_4-L_1P_3-zL_1Q_2-z^2L_1^2\right).$$
If $P_3 \equiv 0$ (i.e. $z$ divides $G_3$), then we already have an equation of type (\ref{eq-q4}). If otherwise $P_3$ does not vanish identically, we claim that it must be the cube of a linear factor because $O$ is not of D type. Hence after a linear change of coordinates involving only $\left(x,y\right)$ we may assume $P_3=y^3$, obtaining an equation of type (\ref{eq-q56}). Indeed, in local analytic coordinates around $O$, $S$ has equation
\begin{equation} \label{eq-Weier}
z^2+\left(P_3\left(x,y\right)+zQ_2\left(x,y\right)\right)+H_4\left(x,y,z\right)=u\left(x,y,z\right)\left(z^2+zA\left(x,y\right)+B\left(x,y\right)\right)
\end{equation}
for some power series $A,B,u$ with $u\left(0,0,0\right)=1$, showing that $S$ is locally around $O$ the double cover of (an open set of) the plane, branched along the curve $\Delta: A^2-4B=0$. Therefore $O$ will be of ADE type if and only if $\Delta$ has a simple singularity. More precisely, $O$ will be of type D if $\Delta$ has a triple point with two or three different tangents.

Analyzing the jets of $A$ and $B$ we obtain
$$A=Q_2+\eta_3\left(x,y\right) + h.o.t., \quad \text{and} \quad B=P_3+\eta_4\left(x,y\right)-P_3\eta_2\left(x,y\right) + h.o.t.,$$
where $\eta_i\left(x,y\right)$ is the coefficient of $z^{4-i}$ in $H_4$. Hence $A^2-4B = -4P_3\left(x,y\right) + h.o.t.$, which shows that $O$ is of type D if $P_3$ is not the cube of a linear form. Assuming from now on that $P_3=y^3$ and expanding $A^2-4B$ a bit further, it can be shown that $\Delta$ has a non-simple singularity if and only if
\begin{equation} \label{cond-not-E}
q_{20}^2-4h_{400}=2q_{20}q_{11}-4h_{310}=q_{20}h_{301}=0.
\end{equation}
If both $q_{20}=h_{301}=0$ then $O$ is not an isolated singularity, hence only one of the two coefficients vanishes. Clearly, if $q_{20}=0$ (resp. $h_{301}=0$) these conditions reduce to (\ref{eq-q5}) (resp. (\ref{eq-q6})).

As for the last claim, suppose that $S$ is given by either (\ref{eq-q4}) or (\ref{eq-q56}), and let $\pi=\left\{ax+by+cz=0\right\}$ be the plane spanned by $\LLL$ and $O$. It can be checked that no plane of the form $\left\{y=\lambda z\right\}$ cuts out a line on $S$, hence we can assume $a=1$, i.e. $\pi$ admits an equation of the form $\widetilde{x}:=x+by+cz=0$. Replacing $x$ by $\widetilde{x}$ introduces higher powers of $z$ in the term linear in $w$, which can be removed with a linear change of $w$ as at the beginning of the proof. This will change $Q_2$ and $H_4$, but the new equation will remain of type (\ref{eq-q4}) (resp. (\ref{eq-q56}) satisfying the conditions (\ref{cond-not-E})).
\end{proof}

We introduce the following notation.
\begin{df} 
We say that the quartic $\XXf$ belongs to the family (Q4) (resp.  (Q5), resp.  (Q6)) iff after a coordinate change it is given by \eqref{eq-q4}
(resp. \eqref{eq-q56}, \eqref{eq-q5}, resp. \eqref{eq-q56}, \eqref{eq-q6}) and the point $O = (0:0:0:1)$ is an isolated singularity of $\XXf$.
\end{df}

If the quartic $\XXf$ belongs to one of the above families, then  
the (reduced) tangent cone $C_{O}\XXf$ is the plane  $\left\{z=0\right\}$.
Obviously, a line $\sLLL \subset \XXf$ that runs through the double point  $O$ is a component of the intersection 
$C_{O}\XXf \cap \XXf$. This results in the following simple observation.

\begin{lem} \label{obs-lines-through-double-point}
a) Every  quartic $\XXf$ in the family (Q4) contains at most four lines through the singularity $O$, which form a hyperplane section of $\XXf$. \\
b) If a surface $\XXf$ belongs to (Q5), then exactly one line $\sLLL_0$ on $\XXf$ contains the singularity $O$. Moreover the hyperplane section $C_{O}\XXf \cap \XXf$
consists of either two lines or the line $\sLLL_0$ and an irreducible conic. If the former holds the line $\sLLL_0$ is met by exactly one other line on $\XXf$. In the latter case, the line  $\sLLL_0$ intersects no other lines on 
the quartic $\XXf$. \\
c) Each quartic $\XXf$ in the family (Q6) contains no lines through the singularity $O$.
\end{lem}
\begin{proof}
a) The intersection $C_{O}\XXf \cap \XXf$ is given by 
the vanishing of $H_4(x,y,0)$,  so $\XXf$ contains at most 4 lines that pass through $O$. \\
b) Put $\sLLL_0 :=\left\{y=z=0\right\}$.  An elementary computation (i.e. substitution in \eqref{eq-q5} and its partials) shows that 
 $\sLLL_0 \subset \XXf$,
the double point $O$ is the only singularity of $\XXf$ on the line $\sLLL_0$ and we have $T_P \XXf =  C_O \XXf = \left\{z=0\right\}$ for all points $P \in \sLLL_0$, $P \neq O$. \\
Moreover, if $h_{220}$ vanishes, then  we have $\XXf.(C_O \XXf) =  3  \sLLL_0 + \ell_1$, where $\ell_1 := \{ h_{130} \cdot x + h_{040} \cdot y +w =0 \}$. Since every line on $\XXf$ that runs through 
a point  $P \in \sLLL_0$, $P \neq O$ lies on the tangent plane $T_P \XXf = \{  z = 0 \}$, the proof (for this case) is complete. \\
Suppose that $h_{220}$ does not vanish, so we can assume  $h_{220}=1$. Then, one can easily see that the hyperplane section $\XXf.(C_O \XXf)$
consists of  the double line $2 \sLLL_0$ and an irreducible conic $C$. Indeed, the latter can be parametrized  (away from the double point $O$) by the map
\begin{equation} \label{eq-parametrization-conic-to-walk-Q5}
\CC \ni t \mapsto P\left(t\right):=\left(t:1:0:-t^2-h_{130}t-h_{040}\right) \in \XXf
\end{equation}
Again, any line on $\XXf$ through a point  $P \in \sLLL_0$, $P \neq O$ must lie on  $C_O \XXf$, so there are no such lines in this case. \\
c) One can easily  see that the hyperplane section  $\XXf.(C_O \XXf)$ contains no lines. 
\end{proof}

Recall that, according to \cite{ctc-wall-99} (see also \cite{degtyarev-90}) (minimal smooth models of) the surfaces in the families (Q4), (Q5) are either rational or elliptic ruled.

Suppose, that a quartic $\XXf$ belongs to the family (Q6). 
Obviously, such a surface meets the hyperplane $\left\{y=tz\right\}$ along a quartic curve with the singular point $O = (0:0:0:1)$. Moreover, 
an elementary computation (see \cite[p.~30]{ctc-wall-99}) shows that  the conditions \eqref{eq-q6}
imply that for generic parameter $t$  the projection from the point $O$ endows the resulting quartic curve with the structure of a double cover of  $\PP^1(\CC)$ branched at exactly  two points.
By Riemann-Hurwitz such a curve is  rational, which yields the following lemma that we will need in the sequel.

\begin{lem} \label{lem-Q6-rational}
{\bf \rm \cite[p.~30]{ctc-wall-99}} If the quartic $\XXf$ belongs to the family (Q6), then $\XXf$ is a rational surface.
\end{lem}

\section{Counting lines on $\XXf$ that meet a given curve} \label{sect-isol-sing}

In this section {\em we assume that the quartic $\XXf \subset  \PP^3(\CC)$
we consider contains finitely many lines.} \\
By \cite[Prop.~1.1]{rs-2012},  (resp. \cite[Thm~1.1]{rs-2014}) a line (resp. a conic) on a smooth quartic
$\XXf \subset  \PP^3(\CC)$ is met by at most $20$ other lines (resp. $48$ lines) on the surface in question. Moreover, both bounds are sharp (see \cite{rs-2012}, \cite{rs-2014}) 
In this section we 
examine the analogous problem  for certain lines and conics on some non-$K3$ quartics.

\begin{lem} \label{lem-Q4}
Let $\XXf$ belong to the family (Q4) and let $\ell_0$ be a line in $\XXf \cap (C_O \XXf)$. Then $\ell_0$ is met by at most seven other lines on  $\XXf$.
\end{lem}
\begin{proof} By Lemma~\ref{lem-coordinates-isolated} we can consider the quartic $\XXf$ that is given by \eqref{eq-q4} with the non-rational double point $O = (0:0:0:1)$.
Moreover, after a linear change of coordinates involving only $x$ and $y$ (hence not changing the shape of the equation \eqref{eq-q4}), we can assume that the line $\ell_0=\left\{x=z=0\right\}$ 
lies on $\XXf$ (i.e. we have $h_{040}=0$ in  \eqref{eq-q4}).  We are to show that  
the line  $\ell_0$  meets at most four  lines on the quartic $\XXf$ that do not contain the singularity  $O$.   \\
Consider the pencil  $|{\mathcal O}_{\XXf}(1) - \ell_0|$, i.e. the  cubics $C_{\lambda}$ residual to $\ell_0$ in the intersection of 
$\XXf$ with  the plane $\Pi_{\lambda} := \left\{x=\lambda z\right\}$, where $\lambda \in \CC$. By direct computation, the intersection of $C_{\lambda}$ with  $\ell_0$ is given by 
\begin{equation} \label{Q4-eq-res-cubic-line}
y^2\left(q_{02}w+\left(h_{031}+\lambda h_{130}\right)y\right)
\end{equation}
Suppose $q_{02}=0$. By direct computation, if  $h_{031}$, $h_{130}$ vanish, then $\ell_0 \subset \sing(\XXf)$, which is excluded by our assumptions. 
Moreover, for $h_{130} = 0$, the equation \eqref{Q4-eq-res-cubic-line} yields that all lines on $\XXf$ that meet $\ell_0$ run through the double point $O$.
If otherwise $h_{130} \neq  0$, we infer from  \eqref{Q4-eq-res-cubic-line} that at most one cubic $C_{\lambda}$ 
can contain  lines on $\XXf$ that meet $\ell_0$ away from $O$.  \\
Thus we have shown  Lemma~\ref{lem-Q4} when  $q_{02}$ vanishes and we can now assume $q_{02}=1$. 
By \eqref{Q4-eq-res-cubic-line} we have
$$
C_{\lambda}.\ell_0 = 2 O + P_{\lambda} \, .
$$
Therefore, the cubic $C_{\lambda}$ is smooth at the point $P_{\lambda}=\left(0:1:0:-h_{031}-\lambda h_{130}\right)$ and the tangent line $T_{P_{\lambda}}C_{\lambda}$ can be the only line in  the hyperplane section $\XXf \cap \Pi_{\lambda}$
that does not contain the double point $O$. For the given parameter $\lambda$ the line in question can be parametrized as $(y:z) \mapsto (\lambda z : y : z : L_{\lambda}\left(y,z\right))$. 
Substituting the above parametrization            
into the equation of $C_{\lambda}$ we arrive at 
a polynomial of the form
\begin{equation} \label{eq-linecounter-q4}
z^2 \left(\qqq_3\left(\lambda\right)y+\qqq_4\left(\lambda\right)z\right),
\end{equation}
where $\deg(\qqq_j) = j$. By definition $T_{P_{\lambda_0}}C_{\lambda_0} \subset \XXf$ iff $\lambda_0$ is a common zero of $\qqq_3$, $\qqq_4$. Since $\XXf$ contains finitely many lines, the polynomials $\qqq_3$ and $\qqq_4$ cannot simultaneously
vanish (identically). Therefore the number of lines in question is bounded by $\deg(\qqq_4)=4$ and the proof is complete. 
\end{proof}

Lemma~\ref{lem-Q4} yields a bound on the number of lines on quartics in the family (Q4).
\begin{pro} \label{pro-Q4}
If $\XXf$ belongs to the family (Q4), then it contains at most 20 lines.
\end{pro}
\begin{proof} By Lemmata~\ref{obs-lines-through-double-point}.a and \ref{lem-Q4} the hyperplane section  $\XXf \cap \left\{z=0\right\}$ consists of at most four lines, each of which is met 
by at most four lines on $\XXf$ that do not contain $O$. 
\end{proof}

The proof of Proposition~\ref{pro-Q4}, gives a direct way of finding all surfaces $\XXf$ in (Q4) that contain exactly $20$ lines. For such a quartic, the hyperplane section $\XXf \cap C_O \XXf$ 
must consist of exactly four lines. For each of the lines the polynomial $\qqq_4$ has four different zeroes and  $\qqq_3 \equiv 0$. 
The latter condition for a given line in the hyperplane section   results in a system of  (huge) equations involving the coefficients of the quartic. 
In this way one obtains the following simple example which shows that 
the bound of Proposition~\ref{pro-Q4} is sharp.
 
\begin{ex} \label{example-Q4-twentylines} Let $\XXf$ be the quartic surface given by 
$$w^2z^2+wz\left(x + y \right)^2+x^3y+xy^3+x^3z-3x^2z^2+4xz^3+3x^2yz-6xyz^2+3xy^2z+4yz^3-3y^2z^2+y^3z.$$
Obviously, $\XXf$ belongs to the family (Q4). Moreover, a standard Gr\"obner basis computation shows that $(0:0:0:1)$ is the unique
singularity of $\XXf$. We claim that 
$\XXf$ contains exactly $20$ lines. 
 
Indeed, the hyperplane section  $\XXf \cap \left\{z=0\right\}$ consists of the four lines 
$$x=0, y=0, x=(\pm i)y.$$
 For the first line (i.e.  $x=0$),
one can easily imitate the proof of Lemma~\ref{lem-Q4} and  check that $\qqq_3 \equiv 0$ whereas $\qqq_4$ has four different roots (c.f. \eqref{eq-linecounter-q4}). This yields that the line in question is met by exactly seven other lines on $\XXf$.
A similar reasoning shows that the line $x=iy$ intersects precisely seven other lines on $\XXf$. Since the quartic $\XXf$ is invariant under the automorphism
$(x:y:z:w) \mapsto (y:x:z:w),$ the claim follows.
\end{ex}

In the following remark we collect certain facts from \cite{rs-2014} that we will use in the sequel. 


\begin{rmk} \label{remark-walk-along-a-conic}
Recall that we assumed the quartic surface  $\XXf$ not to be  ruled by lines.
Given a smooth point $P \in \XXf$ and a line $\ell \subset X$ that runs through the point $P$,  it is clear that $\ell$ is contained both in the tangent space $T_P\XXf$
and in the quadric $V_P$ defined by the Hessian matrix of (the equation) $\XXf$ at $P$. Moreover, 
if  $T_P\XXf$ is no component of  $V_P$ (e.g.  when $\rk(V_P) \geq 3$), then the intersection $T_P \cap V_P$ consists of (at most) two lines (the so-called {\em principal lines}). 
Suppose that the point $P$ varies along a rational curve $C$ that is parametrized by the map $t \mapsto P\left(t\right)$
and put  
\begin{equation} \label{eq-var-of-prip-lines}
Z:= Z_C := \overline{\cup_{t} \left(T_{P(t)} \cap V_{P(t)} \right)}.
\end{equation}
By definition the variety $Z$  contains all the lines on $\XXf$ that meet the rational curve in question. Moreover,   $Z$ is a surface that intersects $\XXf$ properly (indeed, by construction each component of $Z$ is ruled), 
when the condition
\begin{equation} \label{eq-inclusion}
T_P \not\subset V_P \text{ for general } P \in C
\end{equation}
is satisfied. In this case, the number of lines on $\XXf$ that meet the rational curve is bounded by   $4 \deg(Z)$. \\
Obviously, as $t$ varies, $T_{P\left(t\right)}$ and $V_{P\left(t\right)}$ are given by the equations
\begin{equation} \label{eq-trivial-gothic-forms}
{\mathfrak l}_t\left(x_1,\ldots,x_4\right) = \sum_{i=1}^4 \frac{\partial \XXf}{\partial x_i}\left(P\left(t\right)\right) x_i \quad \text{and} \quad {\mathfrak q}_t\left(x_1,\ldots,x_4\right) = \sum_{i=j=1}^4 \frac{\partial^2 \XXf}{\partial x_i \partial x_j}\left(P\left(t\right)\right) x_ix_j.
\end{equation}
and the resultant ${\mathfrak R}_t := Res_t\left({\mathfrak l}_t,{\mathfrak q}_t\right)$ of \eqref{eq-trivial-gothic-forms} with respect to the parameter $t$
belongs to the ideal ${\mathcal I}(Z)$. 
Computing ${\mathfrak R}_t$ by means of the Sylvester matrix leads to
\begin{equation} \label{eq-deg-of-variety-of-principal-lines}
\deg({\mathfrak R}_t) \leq  \deg_t\left({\mathfrak q}_t\right)\deg_x\left({\mathfrak l}_t\right) + \deg_t\left({\mathfrak l}_t\right)\deg_x\left({\mathfrak q}_t\right) = \deg_t\left({\mathfrak q}_t\right)+2\deg_t\left({\mathfrak l}_t\right).
\end{equation}
\end{rmk}

After those preparations we can examine lines on quartics in the family (Q5).

\begin{lem} \label{lem-Q5}
Let $\XXf$ be given by \eqref{eq-q56}, \eqref{eq-q5}, and let $\sLLL_0 \subset \XXf$ be the unique line through $O$. \\
a) If $\XXf.(C_O \XXf) = 2 \sLLL_0 + C$, then  the irreducible conic $C$ is met by at most $27$ lines on $\XXf$. \\
b) Suppose $\XXf.(C_O \XXf) = 3 \sLLL_0 +  \ell_1$. Then the line $ \ell_1$ intersects at most $23$ lines on $\XXf$.
\end{lem}
\begin{proof}
a) Keeping the notation of Lemma~\ref{obs-lines-through-double-point}.b and its proof, we may assume
$h_{220} =1$. In order to count the lines that meet the conic $C$ away from the double
point $O$ we consider the variety $Z_C$ (see \eqref{eq-var-of-prip-lines}). \\  
One can easily see that $\XXf$ is smooth along the image of the parametrization \eqref{eq-parametrization-conic-to-walk-Q5}.
We substitute  \eqref{eq-parametrization-conic-to-walk-Q5} into  the Hessian of $\XXf$
and  check that  
the condition \eqref{eq-inclusion} is fulfilled. 
One can also see  that 
$\deg_t\left({\mathfrak l}_t\right) \leq 3$ and $\deg_t\left({\mathfrak q}_t\right) = 4,$ 
so \eqref{eq-deg-of-variety-of-principal-lines} yields the inequality  $\deg({\mathfrak R}_t) \leq  10$. \\ 
By direct computation, the resultant ${\mathfrak R}_t$  is divisible by $z^2$ and all partials derivatives
of $({\mathfrak R}_t/z^2)$ vanish along the line $\sLLL_0$ and the conic $C$. Thus,both rational curves appear in the intersection cycle
\begin{equation} \label{eq-cycle-Q5}
\XXf \cdot \mbox{V}({\mathfrak R}_t/z^2)
\end{equation}
with multiplicity at least two.  Moreover, since  $\sLLL_0$ is the only line on $\XXf \cap \mbox{V}(z)$  (see Lemma~\ref{obs-lines-through-double-point}.b), we have
$Z_C \subset \mbox{V}({\mathfrak R}_t/z^2)$ and  the number of lines on $\XXf$ that meet $C$ is bounded by $27=32-4-1$. 

\noindent b) Recall that $h_{220} =0$ in this case (see the proof of Lemma~\ref{obs-lines-through-double-point}.b).
 We parametrize the line $\ell_1$ away from the point $(1:0:0:-h_{130}) \in \sLLL_0$ 
and put $ P\left(t\right) := \left(t:1:0:-h_{130}t-h_{040}\right)$ for $t \in \CC$
to study the variety $Z_{\ell_1}$ (see \eqref{eq-var-of-prip-lines}). \\ A direct computation shows that  the quartic $\XXf$ is smooth along the image of 
the parametrization given above
and there are at most two points $P \in \ell_1$ such that  $\rk(V_{P}) < 3$. One can easily check that
$\deg_t\left({\mathfrak l}_t\right) = 3$ and $\deg_t\left({\mathfrak q}_t\right)= 2$. Moreover,  one has 
$z^2 | {\mathfrak R}_t$ and  the  degree-$6$ polynomial  $({\mathfrak R}_t/z^2)$ vanishes along the lines $\sLLL_0$, $\ell_1$. 
In this case, however, it is no longer singular along  $\sLLL_0$ (but its partials vanish along $\ell_1$). Finally, we consider the cycle \eqref{eq-cycle-Q5} to obtain the bound of at most $23$ lines on $\XXf$ that meet $\ell_1$. 
\end{proof}

Recall that in this section all quartic surfaces $\XXf$ are assumed to contain finitely many lines.
As an immediate consequence, we obtain the following proposition.
\begin{pro} \label{prop-Q5}
If a quartic $\XXf$ belongs to the family (Q5), then it contains at most $27$ lines. 
\end{pro}
\begin{proof}
By Lemma~\ref{obs-lines-through-double-point}.b,
every line on $\XXf$ meets the curve residual to the line $\sLLL_0$ in the hyperplane section $\XXf \cap C_O \XXf$.
Therefore, the claim follows directly from Lemma~\ref{lem-Q5}.
\end{proof}

This bound could probably be improved by studying the infinitesimal neighbourhood of $\XXf$ along $C$ (resp. $\ell_1$). However, since it is already better than Rohn's bound for the case of triple points, we will not pursue this improvement.

Finally, we 
 examine lines on quartic surfaces that contain a line of double points.

\begin{lem} \label{lem-count-double-line}
If a quartic $\XXf$ contains a line $\ell_0$ of double points, then $\ell_0$ is  met by at most 16 other  lines on the surface $\XXf$.
\end{lem}
\begin{proof}
We can assume that $\ell_0=\mbox{V}(x,y)$, the quartic  $\XXf$ is given by the  polynomial
\begin{equation} \label{eq-equation-double-line}
Q_{20}\left(x,y,z,w\right) \cdot x^2 + Q_{11}\left(x,y,z,w\right) \cdot xy + Q_{02}\left(x,y,z,w\right) \cdot y^2 = 0
\end{equation}
and the conic $C_0$ residual to the line $\ell_0$ in the hyperplane section $\XXf \cap \mbox{V}(x)$ is irreducible.
Then, all singular conics in $|{\mathcal O}_{\XXf}(1) - 2 \ell_0|$ are given by vanishing of the determinant of the Hessian matrix of 
$(\XXf(x,\lambda x,z,w)/x^2)$. One can easily check that the determinant in question is of degree at most $8$ w.r.t. the parameter $\lambda$, so 
 $\XXf$ contains at most $16$ lines that meet the line $\ell_0$.
\end{proof}

\begin{lem} \label{lem-double-line-residual-conic}
If a quartic $\XXf$ contains a line $\ell_0$ of double points, then a general 
conic $C_0$ in  $|{\mathcal O}_{\XXf}(1) - 2 \ell_0|$  
meets at most $19$ lines on $\XXf$.
\end{lem}
\begin{proof}

We keep the notation of the proof of Lemma~\ref{lem-count-double-line} and consider the variety $Z_{C_0}$ (see \eqref{eq-var-of-prip-lines}). 

Suppose that the conic $C_0$ meets the line $\ell_0$ transversally. 
Without loss of generality we can assume that  the conic in question 
can be parametrized
(away from one point on the double line $\ell_0$) by 
$P\left(t\right):=\left(0:t:1:t^2\right)$ for $t \in \CC$. 
Then  (c.f. \eqref{eq-trivial-gothic-forms}) we have $\deg_t\left({\mathfrak l}_t\right) = 5$  and $\deg_t\left({\mathfrak q}_t\right)= 4$. Moreover, the form ${\mathfrak l}_t$ 
is divisible by $t$ (recall that  $P\left(0\right) \in \sing(\XXf))$. Thus $\tilde{\mathfrak R}_t:=Res_t(\tilde{{\mathfrak l}}_t,{\mathfrak q}_t)$, where $\tilde{{\mathfrak l}}_t :=  {\mathfrak l}_t/t$, belongs to the ideal of the variety $Z_{C_0}$ (c.f. Remark~\ref{remark-walk-along-a-conic})
and \eqref{eq-deg-of-variety-of-principal-lines} gives  $\deg(\tilde{\mathfrak R}_t) \leq 12$. Furthermore, a direct computation shows that 
$x^2 | \tilde{\mathfrak R}_t$, all partials of the polynomial $(\tilde{\mathfrak R}_t/x^2)$ vanish along the line of singularities of $\XXf$ and the non-degeneracy condition \eqref{eq-inclusion} is satisfied.
Therefore, a general conic $C_0 \in |{\mathcal O}_{\XXf}(1) - 2 \ell_0|$ meets the surface  $\mbox{V}(\tilde{\mathfrak R}_t/x^2)$ in at most $18$ points, two of which belong to the line $\ell_0$.
This yields the bound of at most $17$ lines on $\XXf$ that meet the curve $C_0$. 

Suppose now that the conic $C_0$ is tangent to the line $\ell_0$. We can assume that $C_0$ is parametrized (away from the point on the line   $\ell_0$) by
$P(t):=\left(0:1:t^2:t\right)$, where $t \in \CC$. In this case we have $\deg_t({\mathfrak l}_t)\leq 4$ and $\deg_t({\mathfrak q}_t)\leq 4$. Moreover, by a 
direct computation, the resultant ${\mathfrak R}_t$ is again 
divisible by $x^2$, the non-degeneracy condition \eqref{eq-inclusion} is fullfilled and the surface   $\mbox{V}({\mathfrak R}_t/x^2)$  is  singular along the line  $\ell_0$. 
Thus a general conic $C_0 \in |{\mathcal O}_{\XXf}(1) - 2 \ell_0|$ meets at most $19$ lines on $\XXf$. 
\end{proof}

\begin{lem} \label{prop-double-line}
If a quartic $\XXf$ is not ruled by lines and contains a line of double points, then there are at most $35$ lines on $\XXf$.
\end{lem}
\begin{proof}
The claim follows directly from Lemmata~\ref{lem-count-double-line} and \ref{lem-double-line-residual-conic}.
\end{proof}

\begin{ex} \label{example-doubleline-twentyseven}
Consider the quartic $S$ given by \eqref{eq-equation-double-line}, where we put
\begin{eqnarray*}
Q_{20} &:= &wz+yw+yz-10/3\,xy-4\,{y}^{2}, \\
Q_{11} &:= &z (w+z), \\
Q_{02} &:= &1/2\,{w}^{2}+1/4\,wz+1/2\,{z}^{2}+xw+3\,xz+5/4\,yw+5/4\,yz-1/2\,xy.
\end{eqnarray*}
The hyperplane section $\mbox{V}(y)$ splits into $x^2 \cdot z \cdot w$. One can check that
the line of singularities  $\ell_{0}:= \mbox{V}(y,x)$ is met by $16$ other lines on $\XXf$. Moreover, the line  $\mbox{V}(y,z)$ (resp. $\mbox{V}(y,w)$) intersects eight (resp. six) other lines on $S$.
Altogether, the surface $\XXf$ contains $27$ lines. 

\end{ex}


\section{Counting lines on quartic surfaces} \label{sect-counting-lines}

{\sl Recall that in this note all quartic surfaces are assumed not to be cones.} 

In order to count lines on the singular quartics we will use the technique invented already by 
Salmon (see e.g. \cite{salmon2}, \cite[$\S$~13.2]{eisenbud-harris}, \cite[$\S$~8]{Kollar-2014} and the bibliography in the latter).

\begin{lem} {\bf \rm \cite[p. 277]{salmon2}} 
Let $\XXf \subset \PP^3(\CC)$ be a quartic surface  with isolated singularities, none of which is a fourfold point. 
Then there exists a degree-$20$ hypersurface ${\mathcal F}_{\XXf}$ in  $\PP^3(\CC)$
such that $\XXf \nsubseteq {\mathcal F}_{\XXf}$ and 
$$
{\mathcal F}_{\XXf} \cap \XXf = \{P\in S; \text{ there exists  a line } \ell \subset\PP^3 \mbox{ with }  i_P(\XXf,\ell)\geq 4\}.
$$
\end{lem}
\begin{proof} By \cite[$\S$~3]{top-11} no quartic surface $\XXf$  with isolated singularities (none of which is a fourfold point) contains infinitely many lines.
The claim is a standard intersection theory computation (see e.g. \cite[$\S$~13.2]{eisenbud-harris}).            
\end{proof}


\begin{lem} \label{lem-non-K3-isolated-singularities}
Let $\XXf \subset \PP^3(\CC)$ be a non-$K3$ quartic surface with isolated singularities, none of which is a fourfold point. Then $\XXf$ contains at most $48$ lines.
\end{lem}
\begin{proof}
It results from \cite[$\S$~3]{top-11} that the surface $\XXf$ contains only finitely many lines. 
 If $\XXf$ has a triple point, the claim follows from
Lemma~\ref{lem-triple-31}. By Prop.~\ref{pro-Q4},~\ref{prop-Q5} we can assume that $\XXf$ belongs to the family (Q6) and it belongs to neither (Q4) nor (Q5). In particular, 
Lemma~\ref{obs-lines-through-double-point} implies that 
\begin{equation} \label{eq-at-worst-ADE-along-the-line}
\mbox{any line on $\XXf$ runs at worse through rational double points of $\XXf$.}
\end{equation}

Suppose $\XXf$ contains a line $\ell_0$. At first we construct a smooth model of $\XXf$ that we use in the proof. \\ The pencil  $|{\mathcal O}_{\XXf}(1) - \ell_0|$ induces a rational map $\psi_{\ell_0}: \XXf \dashrightarrow \mP^1$ whose general fiber is a smooth elliptic curve. 
Let $\widetilde{\XXf}$ be the minimal resolution of singularities of $\XXf$. By \eqref{eq-at-worst-ADE-along-the-line}, the rational map $\psi_{\ell_0}$  
lifts to a morphism $\widetilde{\psi_{\ell_0}}: \widetilde{\XXf} \rightarrow \mP^1$. 

By Lemma~\ref{lem-coordinates-isolated} we can assume that  $\XXf$ is given by \eqref{eq-q56}, \eqref{eq-q6} and the line $\ell_0$ is contained in the plane $\mbox{V}(x)$.
Let $\widetilde{C}$ be the strict transform of the cubic $C$ that is residual to the line $\ell_0$ in the hyperplane section $\XXf \cap \mbox{V}(x)$. A local computation 
with help of \eqref{eq-q56} shows that $\widetilde{C}$ is a (-1)-curve. Thus the genus-one fibration  $\widetilde{\psi_{\ell_0}}$ is not minimal, but only its fibers induced by cubics through 
non-rational double points of $\XXf$ contain $(-1)$-curves. We blow-down the strict transforms of these cubics to obtain a minimal genus-one fibration 
$\phi_{\ell_0}:\XXf_{\ell_0} \rightarrow \mP^1$. Observe that the smooth surface $\XXf_{\ell_0}$ is rational (see Lemma~\ref{lem-Q6-rational}). Moreover, 
the lines on $\XXf$ that meet (resp. do not intersect) the line $\ell_0$ become components of singular fibers (resp. sections) of  the genus-one fibration  $\phi_{\ell_0}$.

\noindent
{\sl Claim 1. The quartic $\XXf$ contains at most $11$ pairwise disjoint lines.}

\noindent
Indeed, suppose that $\XXf$ contains twelve pairwise disjoint lines $\ell_0, \ell_1,\ldots, \ell_{11}$ and consider the rational elliptic surface $\XXf_{\ell_0}$.
Let $C \subset \XXf$ be a cubic that is coplanar with $\ell_0$ and runs through a non-rational double point of the quartic surface. Since the lines $\ell_1,\ldots, \ell_{11}$ meet $C$ in different points,
their proper transforms  $\widehat{\ell_i} \subset \XXf_{\ell_0}$ meet transversely in exactly $n$ points, where $n$ is the number of non-rational singularities of $\XXf$.
Moreover, each $\widehat{\ell_i}$ is a section of $\phi_{\ell_0}$, hence it is a $\left(-1\right)$-curve. From the equality $\widehat{\ell_i}^2 = \widetilde{\ell_i}^2-n = -2-n$ (where $\widetilde{\ell_i} \subset \widetilde{S}$ denotes the strict transform on the minimal resolution of singularities), we infer that $\XXf$ has one non-rational singularity and that $\widehat{\ell_i}.\widehat{\ell_j}=1$ whenever $i \neq j$. 
In this way we can check that the intersection matrix $[\widehat{\ell_i}.\widehat{\ell_j}]_{i,j=1,\ldots,11}$ has rank $11$, which is impossible because for the Picard number of the  rational elliptic surface $\XXf_{\ell_0}$ one has $\rho(\XXf_{\ell_0})=10$. Contradiction. 

$\mbox{}$ \hfill $\Box_{\mbox{\tiny Claim~1}}$

\noindent
{\sl Claim 2. If a line $\ell_0 \subset \XXf$  is met by at least $12$ other lines, then  $\XXf$ contains at most $22$ lines.}

\noindent
Each line $\ell \subset \XXf$ meeting $\ell_0$ induces a component of a singular fiber of the genus-one fibration on $\XXf_{\ell_0}$. Since   $e(\XXf_{\ell_0})=12$, there can be at most 12 such lines.
Moreover, these  $12$  lines  on $\XXf$ induce four $I_3$-fibers of the fibration $\phi_{\ell_0}$.  If there are no other lines on $\XXf$, then we are done. Otherwise, $\phi_{\ell_0}$ has a section and $\XXf_{\ell_0}$
is the extremal rational surface $X_{3,3,3,3}$. By  \cite[p.~77]{miranda},  the fibration $\phi_{\ell_0}$ has nine sections and the Claim~2 follows.
$\mbox{}$ \hfill $\Box_{\mbox{\tiny Claim~2}}$

\noindent
{\sl Claim 3. If there are no four coplanar  lines on $\XXf$, then each line on  $\XXf$ is met by at most $6$ other lines on the quartic surface.}

\noindent
Indeed, each line that meets the line $\ell_0 \subset \XXf$ gives a contribution of at least $2$ to the Euler number  $e(\XXf_{\ell_0})=12$, so we have at most $6$
such lines. $\mbox{}$ \hfill $\Box_{\mbox{\tiny Claim~3}}$

\noindent
{\sl Claim 4. If there are four coplanar  lines on $\XXf$, then  $\XXf$ contains at most $32$ lines.}

\noindent
By Claim~2 we can assume that each line on $\XXf$ is met by at most $11$ other lines.
Since the sum of Euler numbers of all singular fibers is $12$, the case of exactly $11$ lines that meet a given line on $\XXf$ is also ruled out. 
Since each line on $\XXf$ intersects one of the four coplanar lines,  Claim~4 follows.
$\mbox{}$ \hfill $\Box_{\mbox{\tiny Claim~4}}$

After those preparations we can complete the proof of Lemma~\ref{lem-non-K3-isolated-singularities}. 
By Claims~3 and 4 we can assume that $\XXf$ contains no four coplanar  lines and each line is met by at most $6$ other lines on $\XXf$. \\
If there exist coplanar lines $\ell_1$, $\ell_2$ on $\XXf$ such that the irreducible conic $C \in |{\mathcal O}_{\XXf}(1) - (\ell_1 + \ell_2)|$  
is not contained in the degree-$20$ hypersurface ${\mathcal F}_{\XXf}$, then 
one can easily check that 
the number of lines on $\XXf$ is bounded by 
$$
\deg({\mathcal F}_{\XXf} | _{C \setminus (\ell_1 \cup \ell_2)})+ 6 + 6  = 48.
$$
\noindent Otherwise, each pair of coplanar lines on $\XXf$ induces a conic in the degree-$80$ cycle $({\mathcal F}_{\XXf} \cdot \XXf)$. 
If the number of the residual conics in ${\mathcal F}_{\XXf}$ does exceed $15$ we obtain  again that $\XXf$ contains at most $48$ lines.
Suppose we have at most $15$ conics in  ${\mathcal F}_{\XXf}$. Then there are at least $19$ pairewise disjoint  lines on $\XXf$, which contradicts Claim~1.
\end{proof}

Configurations of lines on 
 irreducible quartic surfaces with non-isolated singularities
 were a subject of interest of classical algebraic geometry.
In the lemma below we recall some well-known facts (see e.g. \cite{dolgachev}, \cite{top-11}). 


\begin{lem} \label{lem-trivial-cases} Let $\XXf \subset \PP^3(\CC)$ be an irreducible quartic.\\
a) If $\XXf$ contains either a twisted cubic of double points or a line of triple points or two skew lines of singular points, then it is ruled by lines.\\
b) {\bf \rm \cite[p.~143]{Clebsch-1868}} If an irreducible quartic $\XXf$ is singular along a conic, then it contains at most  $16$ lines.
\end{lem}
\begin{proof}
a) Observe that if $\XXf$ is singular along a twisted cubic $C_3$, then $\XXf$ is covered by secants of $C_3$. Indeed,
let $P \in \XXf$ be a general (smooth) point, and let $\ell$ be a secant (or tangent) of $C_3$ that runs through the point $P$ (recall that $\PP^3(\CC)$ is the union of  secants and tangents of $C_3$). 
Since $\ell.\XXf \geq 5$, we obtain $\ell \subset \XXf$ and the claim follows.\\
To prove the claim in other cases consider the pencil of hyperplanes that contain the line (resp. one of the lines) of singular points, and observe that each element of the pencil contains another line.

\noindent b) Suppose that $\XXf$ is singular along a smooth conic $C$ and put $\Pi := \mbox{span}(C)$. 
By \cite[$\S$~8.6]{dolgachev}, the system $\left|\caI_C\left(2\right)\right|$ of quadrics in $\mP^3$ induces a rational map
$\phi: \mP^3 \dashrightarrow \mP^4$,
which is not defined at $C$, contracts the plane $\Pi$ to a point $P \in \mP^4$, and is one-to-one outside $\Pi$.
Its image is a smooth quadric threefold $Q_1 \subset \mP^4$. Since $\XXf \in \left|\caI^2_C\left(2\right)\right|$, there exists another 
quadric threefold $Q_2 \subset \mP^4$, such that $\phi(\XXf) = Q_1 \cap Q_2$ (see also \cite[$\S$~7.2.1]{dolgachev}). One can easily check that $\phi$ maps lines on $\XXf$ to lines
 on the Fano surface $Q_1 \cap Q_2$. The latter contains at most $16$ lines (see e.g. \cite[$\S$~8.6.3]{dolgachev}) which yields the claim.
\end{proof}
The above lemma immediately yields the following proposition.

\begin{pro} \label{prop-non-K3--non-isolated-singularities}
Let $\XXf \subset \PP^3(\CC)$ be a non-$K3$ quartic surface. If $\XXf$ is not ruled by lines, then $\XXf$ contains at most $48$ lines.
\end{pro}
\begin{proof} By Lemma \ref{lem-non-K3-isolated-singularities}, we can assume that $\XXf$ has non-isolated singularities. It results from  Bertini (see \cite[Thm~6.3.4]{jouanoulou}) that a (reduced) curve contained in $\sing(\XXf)$ has degree at most 3. By Lemma~\ref{lem-trivial-cases} we can assume that $\XXf$ contains a line of double points. Prop.~\ref{prop-double-line} completes the proof.
\end{proof}


Finally, we can prove the main result of this note. Recall, that all quartic surfaces are assumed not to be cones.

\begin{thm} \label{thm-affine}
a) Let $\XXf \subset \PP^3(\CC)$ be a quartic surface with at most isolated singularities, none of which is a fourfold point. Then $\XXf$ contains at most $64$ lines. \\
b) Let $\XXf \subset \CC^3$ be an affine quartic surface that is not ruled by lines.  Then $\XXf$ contains at most $64$ lines.
\end{thm}
\begin{proof}
a)  If $\XXf$ has at most A-D-E singularities, the claim follows from \cite[Thm~1.1]{veniani-2013}. 
Otherwise, it is not a $K3$-surface, so Prop~\ref{prop-non-K3--non-isolated-singularities} completes the proof. \\
b) Consider the projective closure $\overline{\XXf} \subset \PP^3(\CC)$. If it has at most isolated singularities
we can apply a). Otherwise, we use Prop~\ref{prop-non-K3--non-isolated-singularities}.
\end{proof}

Both bounds are sharp, because the Schur quartic contains $64$ lines (\cite{schur}). The results of Sect.~\ref{sect-isol-sing}
yield the following observation
\begin{obs} \label{obs-much-yields-Q6}
If a non-$K3$ quartic surface $\XXf \subset \PP^3(\CC)$ with isolated singularities and without fourfold points contains more than $31$ lines,
then  $\XXf$ is given by \eqref{eq-q56} and \eqref{eq-q6}. 
\end{obs}

\noindent
Finite field computations combined with results of Sect.~\ref{sect-isol-sing} (in particular Ex.~\ref{example-triple-point})
support the following conjecture.
\begin{conj} \label{conj-31}
a) If a non-$K3$ quartic surface $\XXf \subset \PP^3(\CC)$ contains more than $31$ lines, then it is ruled by lines. \\
b) Let  $\XXf \subset \PP^3(\CC)$ be a quartic surface that is not ruled by lines. If $\sing(\XXf) \neq \emptyset$ then $\XXf$ contains
strictly less than $64$ lines. 
\end{conj}

Finally, recall that the maximal number of lines on quartic hypersurfaces in $\PP^3(\CC)$ with non-empty singular locus (that are not ruled by lines) is unknown.
The example below shows that the number in question is 
attained either by a surface in $(Q6)$ or by a non-smooth $K3$-quartic. 

\begin{ex} \label{example-thirtynine}
One can easily check that the set of singularities of the surface given by
$$
S: x^4 + x z^3 + y^2 z w + x w^3 = 0
$$
consists of  one singular point of type A$_3$ and of three A$_1$-points. 
Studying the lines that meet the plane $\mbox{V}(x)$ one finds that $S$ contains 
exactly $39$ lines.
This is the largest number of lines known on an explicitely given  singular complex quartic surface in $\PP^3(\CC)$ until now. While working on this note we were informed by D.~Veniani
that a Torelli-theorem argument can be applied to show the existence of a complex quartic surface with non-empty singular locus and $40$ lines. 
\end{ex}

\subsection*{Acknowledgement}

We thank Wolf Barth, Alex Degtyarev and Matthias Sch\"utt for helpful discussions.
We are extremely grateful to Duco van Straten for calling our attention to various papers of A.~Clebsch.

\vspace*{1ex}
\noindent
V\'ictor Gonz\'alez-Alonso \\
Institut f\"ur Algebraische Geometrie, Leibniz Universit\"at
   Hannover, Welfengarten 1, 30167 Hannover, Germany \\
gonzalez@math.uni-hannover.de

\vspace*{1ex}
\noindent
S{\l}awomir Rams \\
{\sl Current address:} 
 Institut f\"ur Algebraische Geometrie, Leibniz Universit\"at
   Hannover, Welfengarten 1, 30167 Hannover, Germany \\
{\sl Permanent address:} Institute of Mathematics, Jagiellonian University, 
 ul. {\L}ojasiewicza 6,  30-348 Krak\'ow, Poland \\
slawomir.rams@uj.edu.pl


\end{document}